\documentclass[reqno]{amsart}[12pt]

\usepackage[paper=letterpaper,margin=.9in]{geometry} 
\usepackage[utf8]{inputenc}
\usepackage{lmodern}
\usepackage{array}
\usepackage{adjustbox}
\usepackage{graphicx}
\usepackage{amssymb, amsfonts, amsthm, amsmath}
\usepackage[english]{babel}
\usepackage[pdfborder={0 0 0}]{hyperref}
\usepackage{float}
\usepackage{comment}
\usepackage{tikz}
\usepackage{bbm}
\usetikzlibrary{decorations.markings}
\usetikzlibrary{calc,angles,positioning}
\usetikzlibrary{decorations.pathreplacing}
\usepackage{chngcntr}
\usepackage{apptools}
\AtAppendix{\counterwithin{lemma}{section}}
\AtAppendix{\counterwithin{defin}{section}}
\AtAppendix{\counterwithin{remark}{section}}

\makeatletter \newcommand \listoftodos{\section*{Todo list} \@starttoc{tdo}}
\newcommand\l@todo[2]
{\par \textit{#2}, \parbox{10cm}{#1}\par} \makeatother

\newcommand{\ZZ}{\mathbb{Z}}
\newcommand{\PP}{\mathbb{P}}
\newcommand{\EE}{\mathbb{E}}
\newcommand{\RR}{\mathbb{R}}
\newcommand{\XX}{\mathbb{X}}

\newcommand{\NN}{\mathbb{N}}

\newcommand{\GG}{\mathbb{G}}

\newcommand{\YY}{\mathbb{Y}}

\newcommand{\vp}{\varphi}
\newcommand\numberthis{\addtocounter{equation}{1}\tag{\theequation}}
\newcommand{\paro}{\vec{g}} 
\newcommand{\rparl}{\vec{y}}
\newcommand{\newH}{\widetilde{H}}
\newcommand{\id}{\mathbbm{1}}
\newcommand{\x}{\vec{x}}
\newcommand{\xp}{\vec{x'}}
\newcommand{\xpp}{\vec{x''}}

\newcommand{\xpr}{x'}

\newcommand{\leftvec}[2]{\big(#1 < \cdots < #2\big)}
\newcommand{\rightvec}[2]{(#1 > \cdots > #2)}
\newcommand{\y}{\vec{y}}
\newcommand{\yp}{\vec{y'}}
\newcommand{\ypp}{\vec{y''}}

\newcommand{\zp}{\vec{z'}}

\renewcommand{\wp}{\vec{w'}}

\newcommand{\leftexp}{\EE^{\vec{x}} \big[\newH(\vec{x}(1), \vec{y})\big]}

\newcommand{\onepart}{\mathbf{p}}
\newcommand{\ronepart}{\overleftarrow{\mathbf{p}}}

\theoremstyle{plain}
\newtheorem{theorem}{Theorem}[section]
\newtheorem{defin}[theorem]{Definition}
\newtheorem{prop}[theorem]{Proposition}

\newtheorem{lemma}[theorem]{Lemma}

\newtheorem{remark}[theorem]{Remark}

\begin{document}
\begin{abstract}
We prove that Sch\"{u}tz's ASEP Markov duality functional is also a Markov duality functional for the stochastic six vertex model. We introduce a new method that uses induction on the number of particles to prove the Markov duality.
\end{abstract}
\numberwithin{equation}{section}
\title{Markov Duality for Stochastic Six Vertex Model}
\author{Yier Lin}
\address{Y.\ Lin,
	Department of Mathematics, Columbia University,
	2990 Broadway, New York, NY 10027}
\email{yl3609@columbia.edu}
\maketitle

\section{Introduction}
\subsection{Stochastic six vertex model}
The stochastic six vertex model (S6V model) is a classical model in 2d statistical physics first introduced by Gwa and Spohn \cite{GS92}, as a special case of the six vertex (ice) model (see for example \cite{Lieb74} and \cite{Bax16}). 
We associate each vertex in $\ZZ^2$ with six types of configurations with weights parametrized by $0 < b_1, b_2 <1$, see Figure \ref{fig:vertexconfig}. The configurations chosen for two neighboring vertices need to be compatible in the sense that the lines keep flowing. We consider the lines from the south and the west as the inputs and the lines to the north and the east as the outputs. 
Each vertex is conservative in the sense that the number of input lines equals the number of output lines. The model is stochastic in the sense that when we fix the inputs, the weights of possible configurations sum up to $1$. 
\begin{figure}[ht]
\centering
\begin{adjustbox}{valign=t}
	\begin{tabular}{|c|c|c|c|c|c|c|}
		\hline
		Type & I & II & III & IV & V & VI \\
		\hline
		\begin{tikzpicture}
		\draw[fill][white] (0.5, 0) circle (0.08);
		\draw[very thick][white] (0, 0) -- (1,0);
		\draw[very thick][white] (0.5, -0.5) -- 
		(0.5,0.5);
		\node at (0.5, 0) {Configuration};
		\end{tikzpicture}
		&
		\begin{tikzpicture}
		\draw[fill] (0.5, 0) circle (0.08);
		\draw[very thick] (0, 0) -- (1,0);
		\draw[very thick] (0.5, -0.5) -- (0.5,0.5);
		\end{tikzpicture}
		&
		\begin{tikzpicture}
		\draw[very thick][dashed][white] (0, 0) -- (1,0);
		\draw[very thick][dashed][white] (0.5, -0.5) -- (0.5,0.5);
		\draw[fill] (0.5, 0) circle (0.08);
		\end{tikzpicture}
		&
		\begin{tikzpicture}
		\draw[very thick] (0, 0) -- (1,0);
		\draw[very thick][white][dashed] (0.5, -0.5) -- (0.5,0.5);
		\draw[fill] (0.5, 0) circle (0.08);
		\end{tikzpicture}
		&
		\begin{tikzpicture}
		\draw[very thick] (0, 0) -- (0.5,0);
		\draw[very thick] (0.5, 0) -- (0.5, 0.5);
		\draw[very thick][dashed][white] (0.5, 0) -- (1, 0);
		\draw[very thick][dashed][white] (0.5, -0.5) -- (0.5, 0);
		(0.5,0.5);
		\draw[fill] (0.5, 0) circle (0.08);
		\end{tikzpicture}
		&
		\begin{tikzpicture}
		\draw[very thick][dashed][white] (0, 0) -- (1,0);
		\draw[very thick] (0.5, -0.5) -- (0.5,0.5);
		\draw[fill] (0.5, 0) circle (0.08);
		\end{tikzpicture}
		&
		\begin{tikzpicture}
		\draw[very thick][dashed][white] (0, 0) -- (0.5,0);
		\draw[very thick][dashed][white] (0.5, 0) -- (0.5, 0.5);
		\draw[very thick] (0.5, 0) -- (1, 0);
		\draw[very thick] (0.5, -0.5) -- (0.5, 0);
		(0.5,0.5);
		\draw[fill] (0.5, 0) circle (0.08);
		\end{tikzpicture}
		\\
		\hline
		Weight 
		& 1 & 1 & $b_2$ & $1- b_2$ & $b_1$ & $1-b_1$\\
		\hline
	\end{tabular}
\end{adjustbox}
\caption{Six types of configurations for the vertex.}
\label{fig:vertexconfig}
\end{figure}
\bigskip
\\
The S6V model is a member of the KPZ universality class (see \cite{Cor12, Qua12} for a nice survey). We briefly review a few results that have been obtained for the S6V model. 
\cite{BCG16} proves that under step initial condition, the one point fluctuation of the S6V model height function is asymptotically Tracy-Widom GUE. One point fluctuations of the S6V model under more general initial condition including stationary is obtained in \cite{AB16, Agg16}. In a slightly different direction,  \cite{CGST18} shows that under scaling $\frac{b_1}{b_2} \to 1$ with $b_1$ fixed, the fluctuation of the S6V model height function converges weakly to the solution of KPZ equation. More recently, \cite{BG18, ST18} showed under a different scaling, the height fluctuation of the S6V model converges in finite dimensional distribution to the solution of stochastic telegraph equation. 
\bigskip
\\
In this paper, we consider the S6V model as an interacting particle system (see \cite{GS92} or Section 2.2 of \cite{BCG16}). Consider vertex configurations in the upper half-plane, we restrict ourselves to the boundary condition that there are no lines coming from the left boundary. Then the input lines coming from the bottom of the horizontal axis can be viewed as the trajectories of an exclusion-type particle system. We see the vertical axis as time variable and horizontal axis as  space variable. The vertex configurations compose several paths, which can be viewed as  trajectories of particles with the vertical lines denoting the particle location. As illustrated by Figure \ref{fig:S6Vparticle}, we cut our plane by the line $y = t - \frac{1}{2}$ and the particle location at time $t$ is give by the intersections (red points in the figure) of these trajectories with $y=  t - \frac{1}{2}$. 
To rigorously define our interacting particle system, we first introduce the following state spaces.
\begin{figure}
\begin{tikzpicture}
\draw[very thick][blue][dashed][->] (-1, 0) -- (12, 0);
\draw[very thick][blue][dashed][->] (1, -1) -- (1, 4);
\draw[very thick](2, -0.5) -- (2, 0);
\draw[very thick](4, -0.5) -- (4, 0);
\draw[very thick] (6, -0.5) -- (6, 0.5);
\draw[very thick](8, -0.5) -- (8, 0);
\draw[very thick] (2, 0) -- (4, 0);
\draw[very thick] (4, 0) -- (4, 0.5);
\draw[very thick] (4,0) -- (5,0);
\draw[very thick] (5,0) -- (5, 0.5);
\draw[very thick] (8, 0) -- (9, 0);
\draw[very thick] (9, 0) -- (9, 0.5);
\draw[very thick] (4, 0.5) -- (4, 1.5);
\draw[very thick] (5, 0.5) -- (5,1);
\draw[very thick] (5, 1) -- (6, 1);
\draw[very thick] (6, 1) -- (6, 1.5);
\draw[very thick] (6, 0.5) -- (6, 1);
\draw[very thick](6, 1) -- (9, 1);
\draw[very thick] (9, 0.5) -- (9, 1);
\draw[very thick] (9, 1) -- (10,1);
\draw[very thick] (9, 1) -- (9, 1.5);
\draw[very thick](10, 1) -- (10, 1.5);
\draw[very thick] (4, 1.5) -- (4, 2);
\draw[very thick](4, 2) -- (5, 2);
\draw[very thick] (5, 2) -- (5, 2.5);
\draw[very thick] (6,1.5) -- (6, 2);
\draw[very thick] (6, 2) -- (7,2);
\draw[very thick](7, 2) -- (7, 2.5);
\draw[very thick] (10, 1.5) -- (10, 2);
\draw[very thick] (10, 2) -- (11, 2);
\draw[very thick] (11, 2) -- (11,2.5);
\draw[very thick] (9, 1.5) -- (9, 2.5);
\draw[very thick] (5, 2.5) -- (5, 3);
\draw[very thick] (5, 3) -- (7, 3);
\draw[very thick] (7, 2.5) -- (7, 3);
\draw[very thick] (7 ,3) -- (7,4);
\draw[very thick] (7, 3) -- (9,3);
\draw[very thick] (9,3) -- (9, 4);
\draw[very thick] (9, 2.5) -- (9, 3);
\draw[very thick] (9, 3) -- (10, 3);
\draw[very thick] (10, 3) -- (10, 4);
\draw[very thick] (11, 2.5) -- (11, 3);
\draw[very thick](11, 3) -- (11, 4);
\draw[dashed] (0.5, -0.5) -- (12, -0.5);
\draw[dashed] (0.5, 0.5) -- (12, 0.5);
\draw[dashed] (0.5, 1.5) -- (12, 1.5);
\draw[dashed] (0.5, 2.5) -- (12, 2.5);
\draw[dashed] (0.5, 3.5) -- (12, 3.5);
\node at (0, -0.5) {$t = 0$};
\node at (0, 0.5) {$t = 1$};
\node at (0, 1.5) {$t = 2$};
\node at (0, 2.5) {$t = 3$};
\node at (0, 3.5) {$t = 4$};
\node at (5.5, -1.5) {\textbf{space}};
\node at (-1, 1.75) {\textbf{time}};
\draw[fill][red] (2, -0.5) circle (0.1); 
\draw[fill][red] (4, -0.5) circle (0.1); 
\draw[fill][red] (6, -0.5) circle (0.1); 
\draw[fill][red] (8, -0.5) circle (0.1); 
\draw[fill][red] (4, 0.5) circle (0.1); 
\draw[fill][red] (5, 0.5) circle (0.1); 
\draw[fill][red] (6, 0.5) circle (0.1); 
\draw[fill][red] (9, 0.5) circle (0.1); 
\draw[fill][red] (4, 1.5) circle (0.1); 
\draw[fill][red] (6, 1.5) circle (0.1); 
\draw[fill][red] (9, 1.5) circle (0.1); 
\draw[fill][red] (10, 1.5) circle (0.1); 
\draw[fill][red] (5, 2.5) circle (0.1); 
\draw[fill][red] (7, 2.5) circle (0.1); 
\draw[fill][red] (9, 2.5) circle (0.1); 
\draw[fill][red] (11, 2.5) circle (0.1); 
\draw[fill][red] (7, 3.5) circle (0.1); 
\draw[fill][red] (9, 3.5) circle (0.1); 
\draw[fill][red] (10, 3.5) circle (0.1); \draw[fill][red] (11, 3.5) circle (0.1);
\end{tikzpicture}
\caption{S6V model viewed as an interacting particle system.}
\label{fig:S6Vparticle}
\end{figure}
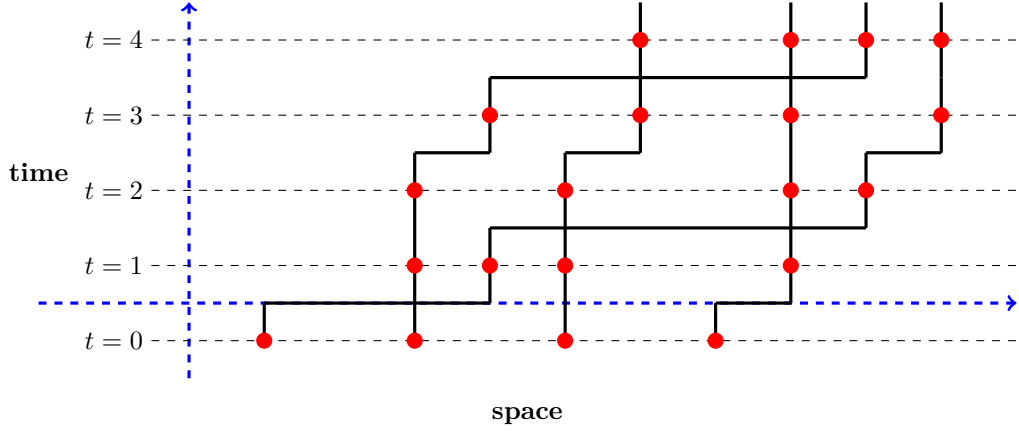 
\begin{defin}\label{def:statespace}
We define the space of left-finite particle occupation configuration $\GG$ to be 
\begin{equation*}
\GG = \left\{\vec{g} = (\cdots, g_{-1}, g_0, g_1, \cdots) \in \left\{0,1\right\}^\ZZ:\  \exists\  i \in \ZZ 
\ \text{so that } g_x =0,\  \forall x \leq i\right\},
\end{equation*} 
where $g_x$ is understood as the number of particles (either zero or one) at location $x$. We also define the space of left-finite particle location configuration to be 
\begin{equation*}
\XX = \left\{\vec{x} = (x_1 < x_2 < \cdots): x_i \in \ZZ \cup \{+\infty\} \text{ for every } i \in \NN \right\},
\end{equation*}
where $x_i$ stands for the location of $i$-th particle counting from the left. Note that there might be infinite or finite number of particles in our particle configuration. For the latter case, there exists some $m \in \NN$ so that $x_i = +\infty$ for $i \geq m$. It is straightforward that there is a bijection $\vp: \XX \to \GG$ defined by 
\begin{equation*}
\vec{g} = \vp (\vec{x}) \text{ such that } g_i = \id_{\left\{\text{there exists } n \in \NN \text{ so that } x_n  =  i\right\}} \text{ for every } i \in \ZZ.
\end{equation*}
\end{defin}
 Having specified our state space, we proceed to define the  particle interpretation of the S6V model as the following discrete-time Markov processes. The following definition is similar as the one that appears in Section 2.1 of \cite{CGST18}.
\begin{defin}\label{def:particlesystem}
We define the S6V location process, which is a discrete-time $\XX$-valued Markov process $\vec{x}(t) = (x_1 (t) < x_2(t) < \cdots )$ with the update rule (transition probability) from $\vec{x}(t)$ to $\vec{x}(t+1)$ specified as follows: 
\bigskip
\\
We denote $x_0(t) = -\infty$ for any $ t \in \ZZ_{\geq 0}$, this is just a convention to simplify the notation. We sequentially consider $i=1,2,\dots $ and update as following independent probabilities:  \bigskip
\\
(a) When $x_i(t) > x_{i-1}(t+1)$, we update $x_{i}(t)$ to $x_{i} (t+1)$ via
\begin{equation*}
\PP\left(x_{i}(t+1) - x_i(t) = n\right) = 
\begin{cases}
b_1, & \text{if} \ n =0;\\
(1 - b_1) (1-b_2) b_2^{n-1}  & \text{if} \ 1 \leq n \leq x_{i+1}(t) - x_{i} (t) -1;\\
(1- b_1) b_2^{n-1} & \text{if} \ n  = x_{i+1}(t) - x_{i} (t);\\
0 & \text{else}; 
\end{cases}
\end{equation*}
(b) When $x_i (t) = x_{i-1} (t+1) $, we update $x_{i}(t)$ to $x_{i} (t+1)$ via
\begin{equation*}
\PP\left(x_{i}(t+1) - x_i(t) = n\right) = 
\begin{cases}
(1-b_2) b_2^{n-1}  & \text{if} \ 1 \leq n \leq x_{i+1} (t) - x_{i} (t) -1;\\
b_2^{n-1} & \text{if} \ n  = x_{i+1}(t) - x_{i} (t); \\
0 & \text{else}.
\end{cases}
\end{equation*}
We also define the S6V occupation process $\vec{g}(t) = (g_x (t))_{x \in \ZZ} \in \GG$ by setting $\vec{g}(t) = \vp(\vec{x}(t))$ i.e. 
\[g_x(t) = \id_{\left\{\text{there exists } n \in \NN \text{ so that } x_n (t) =  x\right\}} \text{ for every } x \in \ZZ.\] Clearly, $\vec{g}(t)$ is a discrete-time $\GG$-valued Markov process. We remark that the occupation  process $\vec{g}(t)$ and location process $\vec{x}(t)$ are just two ways to describe the particle interpretation of the S6V model.
\end{defin}
For a left-finite particle configuration $\vec{g} \in \GG$, we define the height function $N_x (\vec{g})$ to be the total number of particles in this particle configuration that is on the left or at location $x$, i.e.
\begin{equation*}
N_x (\vec{g}) = \sum_{i \leq x} g_i.
\end{equation*}
Our result is a Markov duality between the S6V occupation process and its space reversal, which we define below:
\begin{defin}
Define the space of reversed $k$-particle location configuration $\YY^k = \left\{\vec{y} = (y_1 > \cdots > y_k): \vec{y} \in \ZZ^k \right\}$. The reversed $k$-particle S6V location process $\vec{y}(t) = (y_1 (t) > \cdots > y_k (t))$ is a $\YY^k$-valued Markov process so that $(-y_k(t) < \cdots < -y_1 (t))$ has the same update rule as the S6V location process.
\end{defin}


\subsection{Markov Duality}\label{def:mardual}
\begin{defin}
Given two discrete (continuous) time Markov processes $X(t) \in U$ and $Y(t) \in V$  and a function $H: U \times V \to \RR$, we say that $X(t)$ and $Y(t)$ are dual with respect to $H$ if for any $x \in U, y\in V$ and $t \in \ZZ_{\geq 0}$ ($t \in \RR_{\geq 0}$ for continuous time), we have 
\begin{equation*}
\EE^x\left[H(X(t), y)\right] = \EE^y \left[H(x, Y(t))\right]. 
\end{equation*}
Here we use $\EE^x$ to denote that we take the expectation under initial condition $X(0) = x$. Likewise, $\EE^y$ represents the expectation with initial condition $Y(0) = y$. 
\end{defin}
 
Markov duality has been found for different interacting particle systems including the contact process, voter model and symmetric simple exclusion process, see \cite{Lig12, Lig13}. It also plays an important role in the analysis of models in the KPZ universality class.  The first such example is the asymmetric simple exclusion process (ASEP), which is an interacting particle system on $\ZZ$ with at most one particle at each site. Each particle jumps to the left with rate $\ell$ and jumps to the right with rate $r$. If the site is already occupied by another particle, the jump is excluded.
\bigskip
\\  
We consider ASEP as a process $\vec{g}(t) =  (g_x(t))_{x \in \ZZ} \in \left\{0, 1\right\}^{\ZZ}$, where $g_x (t)$ is an indicator for the event that at time $t$, a particle is at site $x$. We call $\vec{g}(t)$ the ASEP occupation process. When the ASEP has finite $k$-particles, in terms of particle location, we also consider the $k$-particle ASEP location process $\vec{y}(t) = (y_1(t) > \cdots > y_k(t)) \in \YY^k$ where $y_i(t)$ denotes the location of $i$-th particle counting from the right at time $t$.   
\bigskip
\\
Sch\"{u}tz \cite{schutz97} derived the following ASEP duality using a spin chain representation: For any fixed $k \in \NN$,  the ASEP occupation process $\vec{g}(t)$ and the $k$-particle ASEP location process $\vec{y}(t)$ with the jump rate $r$ and $\ell$ reversed are dual with respect to the duality functional
\begin{equation}\label{eq:fdual}
H (\paro, \rparl) = \prod_{i=1}^k g_{y_i} q^{-N_{y_i} (\paro)},
\end{equation}
where $ q = \ell / r$. This generalizes the the  Markov duality satisfied by the symmetric simple exclusion process \cite{Lig12} where $\ell
$ and $r$ are set to be equal. We call \eqref{eq:fdual} Sch\"{u}tz's ASEP Markov duality functional.
\bigskip
\\
\cite{BCS12} uses a different approach to prove Sch\"{u}tz's result by directly applying the Markov generator on the duality functional. Further, they use this method to show that the processes $\vec{g}(t)$ and $\vec{y}(t)$ are also dual with respect to the functional
\begin{equation}\label{eq:sdual}
G (\vec{g}, \vec{y}) = \prod_{i=1}^k q^{-N_{y_i} (\vec{g})}.
\end{equation}
The ASEP is a continous time limit of the S6V model if we scale the parameter by $b_1 = \epsilon \ell$, $b_2 = \epsilon r$ and scale time by $\epsilon^{-1} t$ and shift the space to the right by $\epsilon^{-1} t$, see \cite{BCG16, Agg17}. 
Given the ASEP is the limit of the S6V model and enjoys the Markov duality with respect to the functionals in \eqref{eq:fdual} and \eqref{eq:sdual}, one might wonder if these functionals are the Markov duality functionals for the S6V model as well. Indeed, by setting $q = \frac{b_1}{b_2}$, \cite[Theorem 2.21]{CP16} justifies that the S6V occupation process and the reversed $k$-particle S6V location process are dual with respect
to the functional in \eqref{eq:sdual}\footnote{In fact, \cite[Theorem 2.21]{CP16} proves the duality  for a higher spin generalization of S6V model called stochastic higher spin vertex model.}.
\bigskip
\\
Our main result shows that the S6V model also enjoys a Markov duality with respect to the functional in \eqref{eq:fdual}.
\begin{theorem}\label{thm:main}
Consider the S6V model with parameter $b_1, b_2$ and set $q = \frac{b_1}{b_2}$. For any $k \in \NN$, the S6V occupation process $\vec{g}(t) \in \GG$ and the reversed $k$-particle S6V location process $\vec{y}(t) \in \YY^k$ are dual with respect to the function $H$ given in \eqref{eq:fdual}.
\end{theorem}
\begin{remark}
We remark that the Markov duality in Theorem \ref{thm:main} appeared first in \cite[Theorem 2.23]{CP16} and was later used in proving \cite[Proposition 5.3]{CGST18}. In fact, \cite[Theorem 2.23]{CP16} claims a Markov duality for the stochastic higher spin vertex model (see \cite[Section 2]{CP16} for definition), which is a higher spin generalization of the S6V model (the stochastic higher spin vertex model has vertical and horizontal spin $\frac{I}{2}, \frac{J}{2}$, where $I, J \in \NN$. When $I = J = 1$, it degenerates to the S6V model). However, this Markov duality is false when $I > 1$. In fact, the author of the present paper found a counterexample which is recorded in the erratum \cite{CorwinPetrov2015_erratum}. For the S6V model, the Markov duality holds but the proof of \cite[Theorem 2.23]{CP16} still breaks down\footnote{The proof of \cite[Theorem 2.23]{CP16}
claims that the S6V duality \eqref{eq:fdual} can be deduced by taking the discrete gradient of the Markov duality functional in \eqref{eq:sdual}, which is not true when the number of particles $k$ in Theorem \ref{thm:main} is larger than $1$ \cite{CP18}.}. In this paper, we offer the first correct proof for this Markov duality.   
\end{remark}
\begin{remark}
It appears that the proof of Theorem \ref{thm:main} also adapts to the space inhomogeneous stochastic six vertex model, where we allow the parameters $b_1$, $b_2$ in Figure \ref{fig:vertexconfig} to vary at different locations $x \in \ZZ$ and are expressed by $b_{1,x}$ and $b_{2,x}$. Under the condition that there exists $q > 0$ such that $b_{1,x} = q b_{2,x}$ for all $x \in \ZZ$, Theorem \ref{thm:main} holds for this space inhomogeneous stochastic six vertex model as well. To avoid extra notation, we have opted not to state and prove this more general result here.
\end{remark}
\begin{remark}
It is natural to ask whether our method produces duality for stochastic higher spin vertex model with vertical and horizontal spin $\frac{I}{2}, \frac{J}{2}$. Using fusion, one can prove the same duality in Theorem \ref{thm:main} holds if we take $I = 1$ and $J \in \NN$ (the proof is similar to that of \cite[Corollary 3.2]{CP16}). For $I > 1$, 
as presented in Section \ref{sec:2}, our proof relies heavily on the structure of the duality functional \eqref{eq:fdual} and the particle exclusion property (i.e. at most one particle allows to stay in each location) of the S6V model. It is unclear how to adapt our method proving duality for stochastic higher spin vertex model such as \cite[Theorem 4.10]{kuan18}, since both of the duality functional and the model become more complicated.
\end{remark}
Duality has been obtained for generalization of the ASEP and S6V model using algebraic methods, see \cite{BS151, BS15, CGRS16, kuan16, kuan17, kuan18}. In particular, \cite{CGRS16} proves two ASEP$(q, j)$ (which is a higher spin generalization of ASEP) dualities based on the higher spin representations of $U_q[\mathfrak{sl}_2]$. In the spirit of \cite{CGRS16}, duality has also been proved  for multi-species version of ASEP \cite{BS15, kuan17}. \cite{kuan18} obtains a duality for the multi-species version of the stochastic higher spin vertex model via an algebraic construction. 
Instead of using algebraic tools to prove duality, our proof of Theorem \ref{thm:main}  follows a straightforward induction approach.
\bigskip
\\
We remark that the duality functional from \cite[Theorem 4.10]{kuan18} has a degeneration to S6V model. We state this degeneration here as a lemma. 
\begin{prop}
Consider the S6V model with parameter $b_1, b_2$ and set $q = \frac{b_1}{b_2}$. For any $k \in \NN$, the S6V occupation process $\vec{g}(t) \in \GG$ and the reversed $k$-particle S6V location process $\vec{y}(t) \in \YY^k$ are dual with respect to the functional
\begin{equation}\label{eq:tdual}
D(\vec{g}, \vec{y}) = \prod_{i=1}^k (1-g_{y_i}) q^{-N_{y_i} (\paro)}.
\end{equation}
\end{prop}
We remark that there is a misstatement in  \cite[Theorem 4.10]{kuan18}. The particles in the process $\mathcal{Z}$ and $\mathcal{Z}_{rev}$
were stated to jump to the left and to the right respectively (see pp.164 of \cite{kuan18}). However, after discussing with the author, we realize that the correct statement is that the particles in $\mathcal{Z}$ jump to the right and those in $\mathcal{Z}_{rev}$ jump to the left \cite{kuan}.
\begin{proof}
Taking the spin parameter $m_x = 1$ for all $x \in \ZZ$ and species number $n=1$ and substituting $q$ by $q^{1/2}$, the multi-species higher spin vertex model considered in \cite[Theorem 4.10]{kuan18} degenerates to the stochastic six vertex model (see \cite[Section 2.6.2]{kuan18} for detail). Referring to the duality functional  $\langle \xi | D(u_0) | \eta \rangle$ considered in \cite[Theorem 4.10]{kuan18} (note that $\xi$ is the configuration for the process $\mathcal{Z}$ and $\eta$ is the configuration for the process $\mathcal{Z}_{rev}$). By substituting the configuration 
$\eta$ by $\vec{g} = (g_x)_{x \in \ZZ}$ and the configuration 
$\xi$ by the $k$-particle location configuration  $\vec{y} = (y_1, \cdots, y_k )$, we obtain that the reversed S6V occupation process $\widetilde{g}(t)$ (with particles jumping to the left) is dual to the $k$-particle S6V location process $\widetilde{y}(t)$ (with particles jumping to the right) with respect to the functional 
\begin{equation*}
\widetilde{D}(\vec{g}, \vec{y}) = \prod_{i=1}^k (1- g_{y_i}) q^{-\sum_{z > y_i} g_z - \frac{1}{2} g_{y_i}} = \prod_{i=1}^k (1- g_{y_i}) q^{-\sum_{z \geq y_i} g_z}. 
\end{equation*}
In the last equality we used the fact that $g_{y_i} \in \left\{0, 1\right\}$. Since $\widetilde{g}(t)$ and $\widetilde{y}(t)$ are nothing but the space reversal of $\vec{g}(t)$ and $\vec{y}(t)$ in the lemma. After swapping the role of left and right (then $\sum_{z \geq y_i} g_z$ is exactly the height function $N_{y_i} (\vec{g})$), we readily obtain the duality in \eqref{eq:tdual}.
\end{proof}
When we take $k = 1$ in Theorem \ref{thm:main}, our duality can be simply derived by subtracting the functional $G$ in \eqref{eq:sdual} by the functional $D$ in \eqref{eq:tdual} (see Lemma \ref{lem:inducbasis}). However, it appears that when $k> 1$, there is no easy way to obtain our duality by combining the duality functionals in \eqref{eq:sdual} and \eqref{eq:tdual}. 
\bigskip
\\
Finally, we explain several applications of our duality. Theorem \ref{thm:main} combined with the other S6V duality \eqref{eq:sdual} are the main tools for proving the self-averaging property of the specific quadratic function of the S6V height function in \cite[Proposition 5.3]{CGST18}, which is the crux in proving the convergence of stochastic six vertex model to KPZ equation. In a different direction, by using duality, we can compute the exact moment formula of certain observables of our model. \cite{BCS12} uses the Sch\"{u}tz-type duality of ASEP to derive the moment generating function of the ASEP height function under Bernoulli step initial data. Applying a similar approach, we expect by using our duality and the S6V model Bethe ansatz eigenfunction given by \cite[Proposition 2.12]{CP16}, we can reprove the moment formula appearing in \cite[Theorem 4.12]{BCG16}  and \cite[Theorem 4.4]{AB16}. Since this application of the duality is not related to our paper, we do not pursue to give the proof here.
\subsection{Acknowledgment}
The author wants to thank Ivan Corwin for his helpful discussions and comments pertaining to the earlier draft of this paper and also for his advice on the aspect of paper writing. We wish to thank Promit Ghosal for his useful comments and Jeffrey Kuan for his helpful discussion about the result in his paper. We are also grateful to the anonymous referees for their valuable suggestions. The author was partially supported by the Fernholz Foundation's
``Summer Minerva Fellow" program and also received summer support from Ivan Corwin's
NSF grant DMS:1811143.

\section{Proof of Theorem \ref{thm:main}}\label{sec:2}
In this section, we prove Theorem \ref{thm:main}. We first introduce several notations for our proof. Define the space of $\ell$-particle location configuration
\begin{equation*}
\XX^\ell = \left\{\vec{x} = (x_1 < \cdots < x_\ell):  \x \in \ZZ^\ell \right\}.
\end{equation*}
We also denote by $|\vec{g}|$, $|\vec{x}|$ and $|\vec{y}|$ the number of particles in the particle configuration $\vec{g} \in \GG$, $\vec{x} \in \XX$ and $\vec{y} \in \YY^k$  (obviously $|\vec{y}| = k$ when $\vec{y} \in \YY^k$) respectively.
\bigskip
\\  
Referring to the Definition \ref{def:mardual} of Markov duality, we need to show that for any $k \in \NN$ and under any initial states $\vec{g} \in \GG$, $\vec{y} \in \YY^k$, we have 
\begin{equation*}
\EE^{\vec{g}}\big[H(\vec{g}(t), \vec{y})\big] = \EE^{\vec{y}}\big[H(\vec{g}, \vec{y}(t))\big].
\end{equation*}
By Markov property, it suffices to prove that the preceding equation holds for $t=1$, namely, for any $k \in \NN$ and any $\vec{g} \in \GG$, $\vec{y} \in \YY^k$, we have
\begin{equation}\label{eq:needtoshow}
\EE^{\vec{g}}\big[H(\vec{g}(1), \vec{y})\big] = \EE^{\vec{y}}\big[H(\vec{g}, \vec{y}(1))\big].
\end{equation}

Observing that $|\vec{y}|$ is finite (since $\vec{y} \in \YY^k$) whereas $|\vec{g}|$ can either be finite or infinite. We claim that it suffices to prove \eqref{eq:needtoshow} for all $\vec{g} \in \GG$ such that $|\vec{g}|$ is finite, here is the reason: Suppose we have proved \eqref{eq:needtoshow} for every $\vec{g} \in \GG$ with $|\vec{g}| < \infty$. For $\vec{g} \in \GG$ and $\vec{y} = (y_1 > \cdots > y_k) \in \YY^k$ such that $|\vec{g}| = \infty$, we consider a particle configuration $\vec{g'} \in \GG$ that corresponds with $\vec{g} \in \GG$ in the following way:
\begin{equation*}
g'_i  = 
\begin{cases}
g_i  \qquad &\text{if } i \leq y_1;\\
0 \qquad &\text{if } i > y_1.
\end{cases}
\end{equation*}
Clearly, $|\vec{g'}| < \infty$ and hence
$\EE^{\vec{g'}}\big[H(\vec{g}(1), \vec{y})\big] = \EE^{\vec{y}}\big[H(\vec{g'}, \vec{y}(1))\big].$
Additionally, observing that the particles in the configuration $\vec{g}$ which are on the right of $y_1$ have no contribution to both of the expectations $\EE^{\vec{g}}\big[H(\vec{g}(1), \vec{y})\big]$ and $\EE^{\vec{y}}\big[H(\vec{g}, \vec{y}(1))\big]$ (see Figure \ref{fig:infinite particle}), thus 
\begin{equation*}
\EE^{\vec{g}}\big[H(\vec{g}(1), \vec{y})\big]  = \EE^{\vec{g'}}\big[H(\vec{g}(1), \vec{y})\big], \qquad 
\EE^{\vec{y}}\big[H(\vec{g}, \vec{y}(1))\big]  = \EE^{\vec{y}}\big[H(\vec{g'}, \vec{y}(1))\big].
\end{equation*}
We conclude that $\EE^{\vec{g}}\big[H(\vec{g}(1), \vec{y})\big] = \EE^{\vec{y}}\big[H(\vec{g}, \vec{y}(1))\big]$ also holds for all $\vec{g} \in \GG$ with $|\vec{g}| = \infty$.
\bigskip
\\
\begin{figure}[ht]
\centering
\begin{tikzpicture}
\draw[very thick][->] (0,0) --(12,0);
\draw[very thick][<-] (0,1.3) --(12,1.3);
\draw[fill] (3,0) circle (0.1);
\draw[fill] (2,0) circle (0.1);
\draw[fill][blue] (6,0) circle (0.1);
\draw[fill][blue] (8,0) circle (0.1);
\draw[fill][blue] (9,0) circle (0.1);
\draw[fill] (1,1.3) circle (0.1);
\draw[fill] (3,1.3) circle (0.1);
\draw[fill] (5,1.3) circle (0.1);
\node[below=4pt] at (1,0) {$\dots$};
\node[below=4pt] at (10,0) {$\dots$};
\node[below=4pt] at (1,1.3) {$y_3$};
\node[below=4pt] at (3,1.3) {$y_{2}$};
\node[below=4pt] at (5,1.3) {$y_{1}$};
\node[below=4pt] at (5,1.3) {$y_{1}$};
\node[below=4pt] at (5, 0) {$y_{1}$};
\node[right=4pt] at (12,0) {$\vec{g}$};
\node[right=4pt] at (12,1.3) {$\vec{y}$};
\foreach \x/\xtext in {1,2,3,4,5,6,7,8,9,10,11}
\draw[thick] (\x cm, -0.1)--(\x cm, 0.1); 
\foreach \x/\xtext in {1,2,3,4,5,6,7,8,9,10,11}
\draw[thick] (\x cm, 1.2)--(\x cm, 1.4); 

\end{tikzpicture}
\caption{The picture above shows an example for the initial states $\vec{g} \in \GG$ and $\vec{y} \in \YY^3$. Since the particles in $\vec{g}$ jumps to the right and the particles in $\vec{y}$ jumps to the left (as illustrated by the arrows), the blue particles in $\vec{g}$ (that are on the right of  $y_1$) do not contribute to the computation of $\EE^{\vec{g}}\big[H(\vec{g}(1), \vec{y})\big]$ and $\EE^{\vec{y}}\big[H(\vec{g}, \vec{y}(1))\big]$.}
\label{fig:infinite particle}
\end{figure}
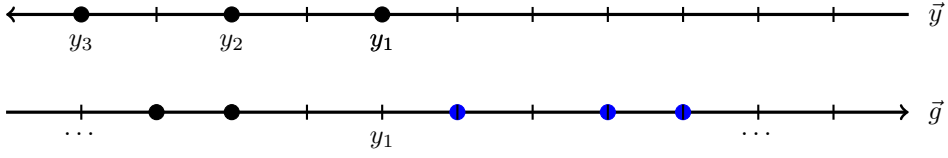
\bigskip
\\
It remains to prove \eqref{eq:needtoshow} when $|\vec{g}|$ is finite. In other words, we need to prove \eqref{eq:needtoshow} for all $\ell, k \in \NN$ and all $\vec{y} \in \YY^k$, $\vec{g} \in \GG$  satisfying $|\vec{g}| = \ell$.   
We apply induction according to $\ell + k$.
The first thing is to show that \eqref{eq:needtoshow} holds when $\min (\ell, k) = 1$, as the induction basis.
\begin{lemma}\label{lem:inducbasis}
When $\min(\ell, k) = 1$, \eqref{eq:needtoshow}  holds.
\end{lemma}
\begin{proof}[Proof of Lemma \ref{lem:inducbasis}]
Since $\min(\ell, k) = 1$, we have either $k > 1$ and $\ell = 1$, or $k=1$.
Note that
\begin{equation*}
H(\vec{g}(1), \vec{y}) = \prod_{i=1}^k g_{y_i}(1) q^{-N_{y_i} (\vec{g}(1))}.
\end{equation*}
If $k > 1$ and $\ell = 1$, since $\vec{g}(1)$ has only one non-zero component and $k > 1$, we have $\prod_{i=1}^k g_{y_i}(1) = 0$ for any $\vec{g}(1)$ and thus 
$\EE^{\vec{g}} \big[H(\vec{g}(1), \vec{y})\big] = 0.$
Similarly, we have $\EE^{\vec{y}}\big[H(\vec{g}, \vec{y}(1))\big] = 0$ hence the desired equality holds. 
\bigskip
\\
If $k = 1$, we note that $G(\vec{g}, \vec{y}) = H(\vec{g}, \vec{y}) - D(\vec{g}, \vec{y})$, where $H$ and $D$ are given by \eqref{eq:sdual} and \eqref{eq:tdual}. Since the subtraction of two duality functionals is still a duality functional (which follows from Definition \ref{def:mardual}), we obtain the desired \eqref{eq:needtoshow}.
\end{proof}
Before explaining how the induction works, we slightly reformulate \eqref{eq:needtoshow}. In order to keep track of the location of the particles, we utilize the S6V location process $\vec{x}(t)$ in Definition \ref{def:particlesystem}. Via the bijection $\varphi: \XX \to \GG$ (see Definition \ref{def:statespace}), we identify a configuration $\vec{g} \in \GG$ with $\vec{x} \in \XX$ and  define the function $\widetilde{H}$ as
\begin{equation*}
\widetilde{H}(\vec{x}, \vec{y}) = H(\varphi(\vec{x}), \vec{y}).
\end{equation*} 
By the relation $\vec{g}(t) = \varphi(\vec{x}(t)) $ between the S6V occupation process $\vec{g}(t)$ and the S6V location process $\vec{x}(t)$, \eqref{eq:needtoshow} can be paraphrased into the following: 
\bigskip
\\
For any $\ell, k \in \NN$ and any initial states $\vec{x} \in \XX^\ell$ and $\vec{y} \in \YY^k$, we have 
\begin{equation}\label{eq:needtoshoww}
\EE^{\vec{x}} \big[\newH(\vec{x}(1), \vec{y})\big] = \EE^{\vec{y}} \big[\newH(\vec{x}, \vec{y}(1))\big].
\end{equation} 

When $\min(\ell, k) = 1$, \eqref{eq:needtoshoww} is established via Lemma \ref{lem:inducbasis}. When $\min(\ell, k) \geq 2$, we define our induction hypothesis as
\begin{equation}\tag{$\mathrm{HYP_{\ell, k}}$} \label{eq:induchypo}
\eqref{eq:needtoshoww} \text{ holds for any } \vec{x} \in \XX^{n} \text{ and } \vec{y} \in \YY^{m} \text{ with } n + m < \ell +k.
\end{equation}
\\
It suffices to prove \eqref{eq:needtoshoww} for any $\vec{x} \in \XX^{\ell}$ and $\vec{y} \in \YY^k$ under \eqref{eq:induchypo}. We briefly explain our strategy: We decompose the LHS expectation of \eqref{eq:needtoshoww} into a combination of the expectations which are in the form of $\EE^{\xp} \big[\newH(\xp(1), \yp)\big]$ with $|\vec{x'}| + |\vec{y'}| < \ell + k $. A similar decomposition  occurs in the RHS expectation. By applying \eqref{eq:induchypo}, we get the desired \eqref{eq:needtoshoww}.
\bigskip
\\
In the sequel, when there is only one particle in the S6V location process, we denote by $\onepart(x, y)$ the one particle transition probability from location $x$ to $y$. Similarly $\ronepart(x, y)$ denotes the one particle transition probability from $x$ to $y$ for the reversed S6V location process. Clearly, $\onepart(x, y) = \ronepart(y, x)$ and 
\begin{equation*}
\onepart(x, y) = 
\begin{cases}
b_1 &\text{ if } y = x;\\
(1-b_1)(1-b_2)b_2^{y - x -1} & \text{ if } y> x;\\
0 & \text{ else. }
\end{cases}
\end{equation*}
In the sequel, we will frequently use the following elementary fact.
\begin{lemma}\label{lem:temp}
Consider S6V location processes
\begin{align*}
&\vec{x}(t) = \big(x_1 (t) < \dots < x_\ell(t)\big) \text{ with initial state } \vec{x} = (x_1 < \dots < x_\ell),\\
&\vec{x'}(t) = \big(x'_1 (t), \dots, x'_{\ell-1} (t)\big) \text{ with initial state } \vec{x'} = (x_2 < \dots < x_\ell),
\end{align*}
and $\vec{y} = (y_1 > \dots > y_k) \in \YY^k$, if $x_1 < y_k$, then 
\begin{equation}\label{eq:temp7}
\EE^{\vec{x}}\big[\newH(\vec{x}(1), \vec{y}) \id_{\{x_1 (1)  = x_1\}}\big] = q^{-k} b_1 \EE^{\vec{x'}}\big[\newH(\vec{x'}(1), \vec{y})\big].
\end{equation}
If $x_1 = y_k$, then 
\begin{equation}\label{eq:temp8}
\EE^{\vec{x}}\big[\newH(\vec{x}(1), \vec{y}) \id_{\{x_1 (1)  = x_1\}}\big] = q^{-k} b_1 \EE^{\vec{x'}}\big[\newH(\vec{x'}(1), \vec{y'})\big]
\end{equation}
where $\vec{y'} = (y_2 > \dots > y_k)$
\end{lemma}
\begin{proof}
We only prove for \eqref{eq:temp7}, the proof of \eqref{eq:temp8} is similar. Knowing $x_1 (1) = y_k$, the particles at $\vec{x'} = (x_{2}, \dots, x_\ell)$ update as an independent S6V model (see Figure \ref{fig:case2a}). 
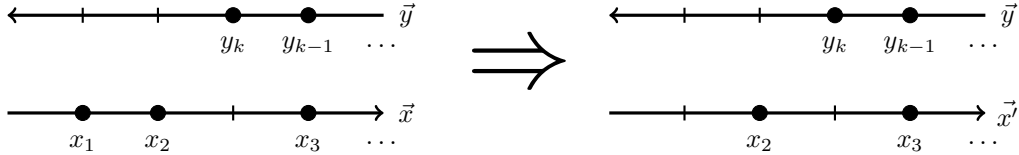
\begin{figure}[ht]
\centering
\begin{tikzpicture}
\draw[very thick][->] (0,-4) --(5,-4);
\draw[very thick][<-] (0,-2.7) --(5,-2.7);
\node at (5.3, -4) {$\vec{x}$};
\node at (5.3, -2.7) {$\vec{y}$};
\foreach \x/\xtext in {1,2,3,4}
\draw[thick] (\x cm, -4.1)--(\x cm, -3.9); 
\foreach \x/\xtext in {1,2,3,4}
\draw[thick] (\x cm, -2.8)--(\x cm, -2.6); 
\draw[fill] (1,-4) circle (0.1);
\draw[fill] (2,-4) circle (0.1);
\draw[fill] (4,-4) circle (0.1);
\draw[fill] (3,-2.7) circle (0.1);
\draw[fill] (4,-2.7) circle (0.1);
\node at (1, -4.4) {$x_{1}$};
\node at (2, -4.4) {$x_{2}$};
\node at (4, -4.4) {$x_{3}$};
\node at (5, -4.4) {$\dots$};
\node at (5, -3.1) {$\dots$};
\node at (3, -3.1) {$y_{k}$};
\node at (4, -3.1) {$y_{k-1}$}; 
\node at (6.8, -3.4) {{\fontsize{40}{48}\selectfont $\Rightarrow$}};
\begin{scope}[xshift = 1cm]
\node at (12.3, -4) {$\vec{x'}$};
\node at (12.3, -2.7) {$\vec{y}$};
\draw[very thick][->] (7,-4) --(12,-4);
\draw[very thick][<-] (7,-2.7) --(12,-2.7);
\foreach \x/\xtext in {8,9,10,11}
\draw[thick] (\x cm, -4.1)--(\x cm, -3.9); 
\foreach \x/\xtext in {8,9,10,11}
\draw[thick] (\x cm, -2.8)--(\x cm, -2.6);
\draw[fill] (9,-4) circle (0.1);
\draw[fill] (11,-4) circle (0.1);
\draw[fill] (10,-2.7) circle (0.1);
\draw[fill] (11,-2.7) circle (0.1);
\node at (9, -4.4) {$x_{2}$};
\node at (11, -4.4) {$x_{3}$};
\node at (11, -3.1) {$y_{k-1}$};
\node at (10, -3.1)  {$y_{k}$};
\node at (12, -4.4) {\dots};
\node at (12, -3.1) {\dots};
\end{scope}
\end{tikzpicture}
\caption{Given $x_1 (1) = x_1$, in one step update procedure, there is no interaction between the particle at $x_1$ and the remaining particles in $\vec{x}$. This allows us to treat the remaining particles  as an independent S6V location process starting from $\vec{x'} = (x_2 < \dots < x_\ell)$. This provides the intuition of \eqref{eq:temp7}.}	
\label{fig:case2a}
\end{figure}
Note that the probability of $x_1 (1) = x_1$ is $b_1$, furthermore, knowing $x_1 (1) = x_1$, we have
\begin{equation*}
\newH(\vec{x}(1), \vec{y}) = q^{-k} \newH\big((x_2 (1), \dots, x_k (1)), \vec{y}\big).
\end{equation*} 
Therefore,  
\begin{align*}
\EE^{\vec{x}}\big[\newH(\vec{x}(1), \vec{y}) \id_{\{x_1 (1) = x_1\}} \big] = q^{-k} b_1 \EE^{\xp} \big[\newH(\xp(1), \y)\big],
\end{align*}
which is the desired \eqref{eq:temp7}.
\end{proof}
\begin{proof}[Proof of Theorem \ref{thm:main}]
We denote the processes in \eqref{eq:needtoshoww} by $\vec{x}(t) = \big(x_1(t) < \cdots < x_\ell (t)\big)$, $\vec{y}(t) = (y_1(t) > \cdots > y_k (t))$ and the initial states by $\vec{x} = (x_1 < \cdots < x_\ell)$, $\vec{y} = (y_1 > \cdots > y_k)$.  We split our proof into different cases depending on the relation fo $\vec{x}$ and $\vec{y}$. 
\bigskip
\\
\textbf{Case (1):} 
$y_k \notin \left\{x_1, \cdots, x_\ell \right\}$ 
\bigskip
\\
Denote by $s$ the positive integer satisfying $x_s < y_k < x_{s+1}$. We consider the S6V location processes   
\begin{align*}
&\xp(t) = \big(\xpr_1(t) < \cdots < x'_{s} (t)\big)  \quad \text{ with initial state } \xp = (x_1 < \cdots < x_s),\\
&\xpp(t) = \leftvec{x''_1 (t)}{x''_{\ell-s} (t)} \text{ with initial state } \xpp = (x_{s+1} < \dots < x_\ell),
\end{align*}
and the reversed S6V location processes, see Figure \ref{fig:case1} (here we do not put arrow on $y'$ since it only has one particle). 
\begin{align*}
&y'(t) \text{ with initial state } y' = y_k, \\
&\ypp(t) = \big(y''_1(t) > \cdots > y''_{k-1} (t)\big)\text{ with initial state } \vec{y''} = (y_1 > \cdots > y_{k-1}).
\end{align*}
Observing $|\xp| + |y'|$ and $|\vec{x''}| + |\vec{y''}|$ are both less than $\ell + k$, hence via \eqref{eq:induchypo}
\begin{equation*}
\EE^{\xp} \big[\newH(\xp(1), y')\big] = \EE^{y'} \big[\newH(\xp, y'(1))\big], \qquad \EE^{\vec{x''}} \big[\newH(\xpp(1), \vec{y''})\big] = \EE^{\vec{y''}} \big[\newH(\vec{x''}, \ypp(1))\big].
\end{equation*}
To prove \eqref{eq:needtoshoww}, it suffices to show 
\begin{align}\label{eq:temp1}
\EE^{\vec{x}}\big[\newH(\vec{x}(1), \vec{y}) \big] = q^{-s (k-1)} \EE^{\xp} \big[\newH(\xp(1), y')\big], \EE^{\vec{x''}} \big[\newH(\xpp(1), \vec{y''})\big],
\\
\label{eq:temp2}
\EE^{\vec{y}} \big[\newH(\vec{x}, \vec{y}(1))\big] = q^{-s (k-1)} \EE^{y'} \big[\newH(\xp, y'(1))\big] \EE^{\vec{y''}} \big[\newH(\vec{x''}, \ypp(1))\big].
\end{align}
\begin{figure}[ht]
\begin{tikzpicture}
\draw[very thick][->] (0,0) --(12,0);
\draw[very thick][<-] (0,1.3) --(12,1.3);
\draw[fill] (3,0) circle (0.1);
\draw[fill] (2,0) circle (0.1);
\draw[fill] (6,0) circle (0.1);
\draw[fill] (8,0) circle (0.1);
\draw[fill] (9,0) circle (0.1);
\draw[fill] (5,1.3) circle (0.1);
\draw[fill] (7,1.3) circle (0.1);
\draw[fill] (9,1.3) circle (0.1);
\node at (1,-0.4) {$\dots$};
\node at (10,-0.4) {$\dots$};
\node[below=4pt] at (5,1.3) {$y_k$};
\node[below=4pt] at (7,1.3) {$y_{k-1}$};
\node[below=4pt] at (9,1.3) {$y_{k-2}$};
\node[below=4pt] at (5, 0) {$y_{k}$};
\node[below=4pt] at (2, 0) {$x_{s-1}$};
\node[below=4pt] at (3, 0) {$x_{s}$};
\node[below=4pt] at (6, 0) {$x_{s+1}$};
\node[below=4pt] at (8, 0) {$x_{s+2}$};
\node[below=4pt] at (9, 0) {$x_{s+3}$};
\node[right=4pt] at (12,0) {$\vec{x}$};
\node[right=4pt] at (12,1.3) {$\vec{y}$};
\foreach \x/\xtext in {0,1,2,3,4,5,6,7,8,9,10,11}
\draw[thick] (\x cm, -0.1)--(\x cm, 0.1); 
\foreach \x/\xtext in {1,2,3,4,5,6,7,8,9,10,11,12}
\draw[thick] (\x cm, 1.2)--(\x cm, 1.4); 
\node at (6, -1.5) {{\fontsize{40}{48}\selectfont $\Downarrow$}};
\draw[very thick][->] (0,-4) --(5,-4);
\draw[very thick][<-] (0,-2.7) --(5,-2.7);
\node at (5.3, -4) {$\vec{x'}$};
\node at (5.3, -2.7) {$y'$};
\foreach \x/\xtext in {0,1,2,3,4}
\draw[thick] (\x cm, -4.1)--(\x cm, -3.9); 
\foreach \x/\xtext in {1,2,3,4,5}
\draw[thick] (\x cm, -2.8)--(\x cm, -2.6); 
\draw[very thick][->] (7,-4) --(12,-4);
\draw[very thick][<-] (7,-2.7) --(12,-2.7);
\node at (12.3, -4) {$\vec{x''}$};
\node at (12.3, -2.7) {$\vec{y''}$};
\foreach \x/\xtext in {8,9,10,11}
\draw[thick] (\x cm, -4.1)--(\x cm, -3.9); 
\foreach \x/\xtext in {8,9,10,11}
\draw[thick] (\x cm, -2.8)--(\x cm, -2.6);
\draw[fill] (1,-4) circle (0.1);
\draw[fill] (2,-4) circle (0.1);
\draw[fill] (4,-2.7) circle (0.1);
\node at (1, -4.4) {$x_{s-1}$};
\node at (2, -4.4) {$x_{s}$};
\node at (0, -4.4) {$\dots$};
\node at (10, 0.9) {$\dots$};
\node at (4, -3.1) {$y_k$};
\draw[fill] (8,-4) circle (0.1); 
\draw[fill] (10,-4) circle (0.1);
\draw[fill] (11,-4) circle (0.1);
\draw[fill] (9,-2.7) circle (0.1);
\draw[fill] (11,-2.7) circle (0.1);
\node at (8, -4.4) {$x_{s+1}$};
\node at (10, -4.4) {$x_{s+2}$};
\node at (11, -4.4) {$x_{s+3}$};
\node at (11, -3.1) {$y_{k-2}$};
\node at (9, -3.1)  {$y_{k-1}$};
\node at (12, -4.4) {\dots};
\node at (12, -3.1) {\dots};
\end{tikzpicture}
\caption{By the form of duality functional \eqref{eq:fdual}, we see that $H(\vec{x}(1), \vec{y})$ is zero unless $x_s (1) = y_k$. Knowing $x_s(1) = y_k$, in one step update procedure there is no interaction between the particles at $\vec{x'} = (x_1, \dots, x_s)$ and those at $\vec{x''} = (x_{s+1}, \dots, x_\ell)$ (since $x_s(1) < x_{s+1}$). Hence, we treat the particles at $\vec{x'} = (x_{s+1}, \dots, x_\ell)$ as an independent S6V location process.
This provides the intuition of \eqref{eq:tempp1}.}
\label{fig:case1}
\end{figure}
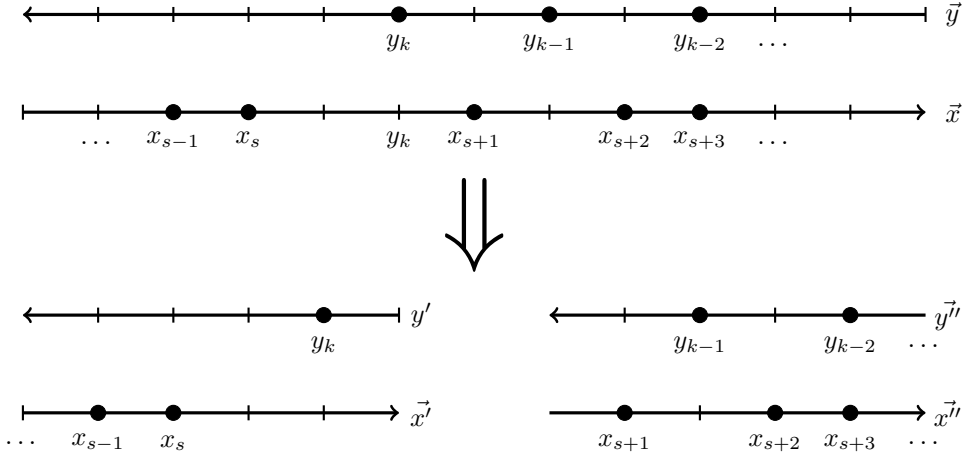
\bigskip
\\
We first show \eqref{eq:temp1}. Observing that by the update rule of $\vec{x}(t)$ defined in Definition \ref{def:particlesystem}, $H(\vec{x}(1), \vec{y}) = 0$ if $x_s(1) \neq y_k$, thus 
\begin{equation*}
\EE^{\vec{x}}\big[\newH(\vec{x}(1), \vec{y}) \big] = \EE^{\vec{x}}\big[\newH(\vec{x}(1), \vec{y}) \id_{\left\{x_s (1) = y_k\right\}}\big].
\end{equation*}
Knowing $x_s (1) = y_k$, we can treat the particles at $\vec{x''} = (x_{s+1}< \dots < x_\ell)$ as an independent S6V model. Therefore, it is straightforward that 
\begin{equation}\label{eq:tempp1}
\EE^{\vec{x}}\big[\newH(\vec{x}(1), \vec{y}) \id_{\left\{x_s (1) = y_k\right\}}\big] = q^{-s (k-1)} \EE^{\xp}\big[\newH(\xp(1), y') \id_{\{x'_s(1) = y_k\}}\big] \EE^{\xpp}\big[\newH(\xpp(1), \vec{y''})\big]. 
\end{equation}
The factor $q^{-s(k-1)}$ comes from knowing $x_s(1) = y_k$, 
\begin{equation*}
\newH(\vec{x}(1), \vec{y}) = q^{-s(k-1)} \newH\big((x_1(1), \dots, x_s(1)), y'\big) \newH\big((x_{s+1}(1), \dots, x_\ell(1)), \vec{y''}\big)
\end{equation*}
Using the fact that $\EE^{\xp}\big[\newH(\xp(1), y') \id_{\{x_s(1) = y_k\}}\big] = \EE^{\xp}\big[\newH(\xp(1), y') \big]$ in \eqref{eq:tempp1}, we conclude \eqref{eq:temp1}.\\ 
Likewise, to show  \eqref{eq:temp2}, we have 
\begin{align*}
\EE^{\vec{y}} \big[\newH(\vec{x}, \vec{y}(1))\big] &= \EE^{\vec{y}} \big[\newH(\vec{x}, \vec{y}(1)) \id_{\left\{y_{k-1} (1) \geq x_{s+1}\right\}}\big] 
=  q^{-s (k-1)} \EE^{y'}\big[\newH(\xp, y'(1))\big] \EE^{\vec{y''}}\big[\newH(\vec{x''}, \ypp(1)) \id_{\{y''_{k-1}(1) \geq x_{s+1}\}}\big].  
\end{align*}
Using the fact that $$\EE^{\vec{y''}}\big[\newH(\vec{x''}, \ypp(1)) \id_{\{y''_{k-1}(1) \geq x_{s+1}\}}\big] = \EE^{\vec{y''}}\big[\newH(\vec{x''}, \ypp(1))\big],$$ we conclude \eqref{eq:temp2}.
\bigskip
\\
\textbf{Case (2):} $y_k \in \left\{x_1, \cdots, x_\ell \right\}$
\bigskip
\\
We divide our discussion into three sub-cases.
\bigskip
\\
\textbf{Case (2a):} $y_k = x_1.$
\bigskip
\\
In this case, let us consider the  S6V location process and the reversed S6V location process 
\begin{align*}
&\xp(t) = (\xpr_1(t) < \cdots < x'_{\ell-1} (t) )  \text{ with initial state } \xp = (x_2 < \cdots < x_\ell),\\
&\yp(t) = \rightvec{y'_1 (t)}{y'_{k-1} (t)} \text{ with initial state } \yp = \rightvec{y_1}{y_{k-1}}.
\end{align*}
Since $|\xp| + |\yp| < \ell + k$,
the induction hypothesis \eqref{eq:induchypo} gives $ \EE^{\xp} \big[\newH(\xp(1), \yp)\big] = \EE^{\yp} \big[\newH(\xp, \yp(1))\big]$. To prove \eqref{eq:needtoshoww}, it suffices to show 
\begin{align}
\label{eq:temp3}
\EE^{\vec{x}}\big[\newH(\vec{x}(1), \vec{y}) \big] &= q^{-k} b_1 \EE^{\xp} \big[\newH(\xp(1), \yp)\big].\\
\label{eq:temp4}
\EE^{\vec{y}} \big[\newH(\vec{x}, \vec{y}(1))\big] 
&= q^{-k} b_1 \EE^{\yp} \big[\newH(\xp, \yp(1))\big].
\end{align}
We first justify \eqref{eq:temp3}.
Since $x_1  = y_k$ and $\vec{x}(t)$ starts from $\vec{x}$, it follows from  the update rule that $H(\vec{x}(1), \vec{y}) = 0$ unless $x_1 (1) = x_1$. Thus, 
\begin{equation*}
\EE^{\vec{x}}\big[\newH(\vec{x}(1), \vec{y}) \big] = \EE^{\vec{x}}\big[\newH(\vec{x}(1), \vec{y}) \id_{\{x_1 (1) = x_1\}} \big] 
\end{equation*}
Using Lemma \ref{lem:temp}, we conclude \eqref{eq:temp3}. Under the same reasoning,  
\begin{align*}
\EE^{\vec{y}} \big[\newH(\vec{x}, \vec{y}(1))\big] = \EE^{\vec{y}} \big[\newH(\vec{x}, \vec{y}(1)) \id_{\{y_k(1) = y_k\}}\big]
= q^{-k} b_1 \EE^{\yp} \big[\newH(\xp, \yp(1))\big],
\end{align*}
which concludes \eqref{eq:temp4}.
\bigskip
\\
\textbf{Case (2b):} $y_k = x_2 > x_1$
\bigskip
\\
The proof for this case is more involved than the previous ones.
We consider the S6V location processes and reversed S6V location process (see Figure \ref{fig:case2b})
\begin{align*}
&\xp(t) = (\xpr_1(t) < \cdots < x'_{\ell - 1} (t) )  \ \text{ with initial state } \xp = (x_2 < \cdots < x_\ell),\\
&\xpp(t) = \leftvec{x''_1 (t)}{x''_{\ell-2} (t)} \, \text{ with initial state } \xpp = \leftvec{x_3}{x_\ell},\\
&\yp(t) = \rightvec{y'_1(t)}{y'_{k-1} (t)} \ \, \text{ with initial state } \yp = \rightvec{y_1}{y_{k-1}}.
\end{align*}
To simplify our notation, we denote 
\begin{align*}
&L_1 = \EE^{\xp} \big[\newH(\xp(1), \y)\big], \qquad L_2 = \EE^{\xp} \big[\newH(\xp(1), \yp)\big], \qquad
L_3 = \EE^{\xpp} \big[\newH(\xpp(1), \yp)\big],\\
&R_1 = \EE^{\y} \big[\newH(\xp, \vec{y}(1))\big], \qquad R_2 = \EE^{\yp} \big[\newH(\xp, \yp(1))\big], \qquad R_3 = \EE^{\yp} \big[\newH(\xpp, \yp(1))\big].
\end{align*}
Since $|\xp| + |\y|, |\xp| + |\yp|, |\xpp| + |\yp|$ are all less than $\ell + k$, we have by induction hypothesis \eqref{eq:induchypo}
\begin{equation} \label{eq:2bidhypo}
L_1 = R_1, \qquad L_2 = R_2, \qquad L_3 = R_3.
\end{equation}
To prove \eqref{eq:needtoshoww}, it suffices to show that 
\begin{align}
\label{eq:temp5}
\EE^{\vec{x}} \big[\newH(\vec{x}(1), \vec{y})\big] = q^{-k} L_1 + b_2^{x_2 - x_1 -1} (q^{-k} L_2 + q^{-(2k-1)} (b_1 b_2 - b_1 - b_2) L_3), \\
\label{eq:temp6}
\EE^{\vec{y}} \big[\newH(\vec{x}, \vec{y}(1)) \big] = 
q^{-k} R_1 + b_2^{x_2 - x_1 -1} (q^{-k} R_2 + q^{-(2k-1)} (b_1 b_2 - b_1 - b_2)R_3).
\end{align}
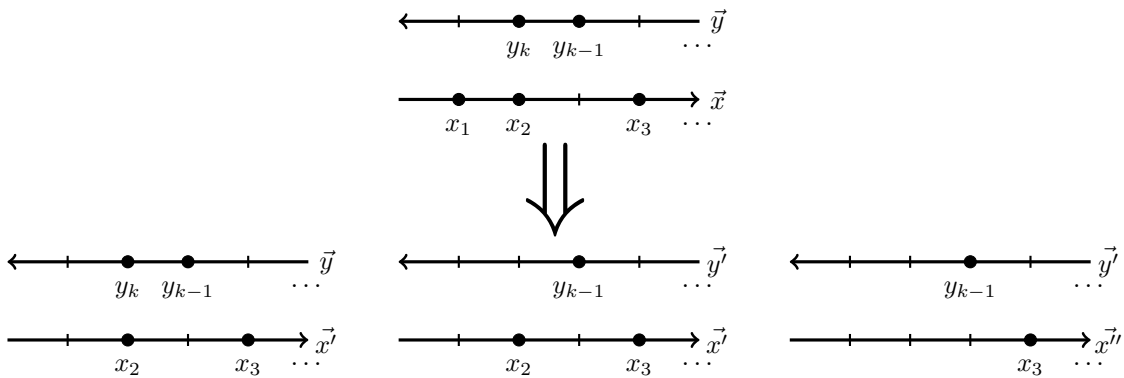
\begin{figure}[ht]
\centering 
\begin{tikzpicture}[scale = 0.8]
\draw[very thick][->] (1,0) --(6,0);
\draw[very thick][<-] (1,1.3) --(6,1.3);
\node at (6.3, 1.3) {$\vec{y}$};
\node at (6.3, 0) {$\vec{x}$};
\draw[fill] (3,0) circle (0.1);
\draw[fill] (2,0) circle (0.1);
\draw[fill] (5,0) circle (0.1);
\draw[fill] (3, 1.3) circle (0.1);
\draw[fill] (4, 1.3) circle (0.1);
\node at (6,-0.4) {$\dots$};
\node at (6,0.9) {$\dots$};
\node[below=4pt] at (3,1.3) {$y_k$};
\node[below=4pt] at (4,1.3) {$y_{k-1}$};
\node[below=4pt] at (5, 0) {$x_{3}$};
\node[below=4pt] at (2, 0) {$x_{1}$};
\node[below=4pt] at (3, 0) {$x_{2}$};
\foreach \x/\xtext in {2,3,4,5}
\draw[thick] (\x cm, -0.1)--(\x cm, 0.1); 
\foreach \x/\xtext in {2,3,4,5}
\draw[thick] (\x cm, 1.2)--(\x cm, 1.4); 
\node at (3.6, -1.5) {{\fontsize{40}{48}\selectfont $\Downarrow$}};
\begin{scope}[xshift = -6.5cm, yshift = -4cm]
\draw[very thick][->] (1,0) --(6,0);
\draw[very thick][<-] (1,1.3) --(6,1.3);
\draw[fill] (3,0) circle (0.1);
\draw[fill] (5,0) circle (0.1);
\draw[fill] (3, 1.3) circle (0.1);
\draw[fill] (4, 1.3) circle (0.1);
\node at (6,-0.4) {$\dots$};
\node at (6,0.9) {$\dots$};
\node at (6.3, 1.3) {$\vec{y}$};
\node at (6.3, 0) {$\vec{x'}$};
\node[below=4pt] at (3,1.3) {$y_k$};
\node[below=4pt] at (4,1.3) {$y_{k-1}$};
\node[below=4pt] at (5, 0) {$x_{3}$};
\node[below=4pt] at (3, 0) {$x_{2}$};
\foreach \x/\xtext in {2,3,4,5}
\draw[thick] (\x cm, -0.1)--(\x cm, 0.1); 
\foreach \x/\xtext in {2,3,4,5}
\draw[thick] (\x cm, 1.2)--(\x cm, 1.4); 
\end{scope}

\begin{scope}[xshift = 0cm, yshift = -4cm]
\draw[very thick][->] (1,0) --(6,0);
\draw[very thick][<-] (1,1.3) --(6,1.3);
\draw[fill] (3,0) circle (0.1);
\draw[fill] (5,0) circle (0.1);
\draw[fill] (4, 1.3) circle (0.1);
\node at (6,-0.4) {$\dots$};
\node at (6,0.9) {$\dots$};
\node[below=4pt] at (4,1.3) {$y_{k-1}$};
\node[below=4pt] at (5, 0) {$x_{3}$};
\node[below=4pt] at (3, 0) {$x_{2}$};
\foreach \x/\xtext in {2,3,4,5}
\draw[thick] (\x cm, -0.1)--(\x cm, 0.1); 
\foreach \x/\xtext in {2,3,4,5}
\draw[thick] (\x cm, 1.2)--(\x cm, 1.4); 
\node at (6.3, 1.3) {$\vec{y'}$};
\node at (6.3, 0) {$\vec{x'}$};
\end{scope}

\begin{scope}[xshift = 6.5cm, yshift = -4cm]
\draw[very thick][->] (1,0) --(6,0);
\draw[very thick][<-] (1,1.3) --(6,1.3);
\draw[fill] (5,0) circle (0.1);
\draw[fill] (4, 1.3) circle (0.1);
\node at (6,-0.4) {$\dots$};
\node at (6,0.9) {$\dots$};
\node[below=4pt] at (4,1.3) {$y_{k-1}$};
\node[below=4pt] at (5, 0) {$x_{3}$};
\foreach \x/\xtext in {2,3,4,5}
\draw[thick] (\x cm, -0.1)--(\x cm, 0.1); 
\foreach \x/\xtext in {2,3,4,5}
\draw[thick] (\x cm, 1.2)--(\x cm, 1.4); 
\node at (6.3, 1.3) {$\vec{y'}$};
\node at (6.3, 0) {$\vec{x''}$};
\end{scope}
\end{tikzpicture}
\caption{When $x_2 = y_k$, we consider three pairs of S6V models starting with less than $|\vec{x}| + |\vec{y}| = \ell + k$ number of particles, we prove  \eqref{eq:needtoshoww} via expressing $\EE\big[\newH(\vec{x}(1), \vec{y})\big]$ (resp. $\EE\big[\newH(\vec{x}, \vec{y}(1))\big]$) in terms of $L_1, L_2, L_3$ (resp. $R_1, R_2, R_3$) and using the induction hypothesis \eqref{eq:2bidhypo}.}
\label{fig:case2b}
\end{figure}
\bigskip
\\
Let us show \eqref{eq:temp5} first. Since $x_2 = y_k$, by Lemma \ref{lem:temp}, 
\begin{equation}\label{eq:2brelation}
L_1 = b_1 q^{-k} L_3. 
\end{equation}
Expanding the LHS expectation of \eqref{eq:needtoshoww} as following (according to the update rule, $x_1(1)$ can not exceed $x_2$)
\begin{equation}\label{eq:2be}
\EE^{\vec{x}} \big[\newH(\vec{x}(1), \vec{y})\big] = \EE^{\vec{x}} \big[\newH(\vec{x}(1), \vec{y}) \id_{\left\{x_1 (1) < x_2\right\}}\big] + \EE^{\vec{x}} \big[\newH(\vec{x}(1), \vec{y}) \id_{\left\{x_1 (1) = x_2\right\}}\big].
\end{equation}
For the first term on the RHS of \eqref{eq:2be}, given $x_1 (1) < x_2$, the particles at $x_2, \dots, x_\ell$ update as an independent S6V model.
Using the fact that $$\PP(x_1 (1) < x_2) = \sum_{y = x_1}^{x_2 - 1} \onepart(x, y ) = 1 - (1-b_1) b_2^{x_2 - x_1 - 1}$$ and knowing $x_1 (1) < x_2$
\begin{equation*}
\newH(\vec{x}(1). \vec{y}) = q^{-k} \newH\big((x_2 (1), \dots, x_\ell (1)), \vec{y}\big),
\end{equation*}
one has 
\begin{align*}
\numberthis \label{eq:2bfirst}
\EE^{\vec{x}} \big[\newH(\vec{x}(1), \vec{y}) \id_{\left\{x_1 (1) < x_2\right\}} \big]
=(1 -  (1 - b_1) b_2^{x_2 -  x_1 - 1}) q^{-k} L_1.
\end{align*}
Let us compute the second term on the RHS of \eqref{eq:2be},  
\begin{equation}\label{eq:2bes}
\EE^{\vec{x}} \big[\newH(\vec{x}(1), \vec{y}) \id_{\left\{x_1 (1) = x_2\right\}}\big] =\sum_{\substack{\vec{z} = (z_1 < \cdots < z_\ell)\\ z_1 = x_2}} \PP^{\vec{x}} \big(\vec{x}(1) = \vec{z}\big) \newH(\vec{z}, \vec{y}).
\end{equation}
For  $\zp = (z_2 < \cdots < z_\ell)$ and $\vec{z} = (z_1 < \cdots < z_\ell)$ with $z_1 = x_2$, we have
\begin{equation}\label{eq:2bess}
\begin{split}
\PP^{\vec{x}}\big(\vec{x}(1) = \vec{z}\big) &= b_2^{x_2 - x_1 - 1} \PP^{\xp}\big(\xp(1) = \zp\big),\\
\newH(\vec{z}, \vec{y}) &= q^{-k} \newH(\zp, \yp).
\end{split}
\end{equation}
Plugging the expression on the RHS of \eqref{eq:2bess} into \eqref{eq:2bes} gives 
\begin{align*}
\EE^{\vec{x}} \big[\newH(\vec{x}(1), \vec{y}) \id_{\left\{x_1 (1) = x_2\right\}}\big] &=  q^{-k} b_2^{x_2  - x_1 - 1} \sum_{\substack{\zp = (z_2< \cdots< z_\ell)\\ z_2 > x_2}}  \PP^{\xp} \big(\xp(1) = \zp\big) \newH(\zp, \yp),\\ 
&= q^{-k} b_2^{x_2 - x_1 - 1} \EE^{\xp} \big[\newH(\xp(1), \yp) \id_{\left\{x'_1 (1) \neq x_2 \right\}}\big].
\end{align*}
Consequently, 
\begin{align*}
\numberthis \label{eq:2bees}
\EE^{\vec{x}} \big[\newH(\vec{x}(1), \vec{y}) \id_{\left\{x_1 (1) = x_2\right\}}\big] = q^{-k} b_2^{x_2 - x_1 - 1} \bigg(L_2 - \EE^{\xp} \big[\newH(\xp(1), \yp) \id_{\left\{x'_1(1)  = x_2\right\}}\big]\bigg).
\end{align*}
It is straightforward that 
\begin{align*}
\EE^{\xp} \big[\newH(\xp(1), \yp) \id_{\left\{x'_1(1)  = x_2\right\}}\big] = q^{-(k-1)} b_1 \EE^{\xpp} \big[\newH (\xpp(1), \yp)\big] = q^{-(k-1)} b_1 L_3.
\end{align*}
Substituting this back to \eqref{eq:2bees} gives
\begin{align*}
\numberthis \label{eq:2bsec}
\EE^{\vec{x}} \big[\newH(\vec{x}(1), \vec{y}) \id_{\left\{x_1 (1) = x_2\right\}}\big] &= q^{-k} b_2^{x_2 - x_1 - 1} L_2  - q^{-(2k-1)} b_1 b_2^{x_2 - x_1 - 1} L_3.
\end{align*}
Note that the LHS of \eqref{eq:2bfirst} and \eqref{eq:2bsec} are the first and second terms on the RHS of \eqref{eq:2be}, one has
\begin{align*}
\EE^{\vec{x}} \big[\newH(\vec{x}(1), \vec{y})\big] &= (1 - (1- b_1) b_2^{x_2 -  x_1 - 1}) q^{-k} L_1 +( q^{-k} b_2^{x_2 - x_1 - 1} L_2  - q^{-(2k-1)} b_1 b_2^{x_2 - x_1 - 1} L_3),\\ 
&= q^{-k} L_1 + b_2^{x_2 - x_1 -1} (q^{-k} L_2 + q^{-(2k-1)} (b_1 b_2 - b_1 - b_2) L_3).
\end{align*}
Therefore, we conclude \eqref{eq:temp5}. Note that in the last line above we used the $L_1 = b_1 q^{-k} L_3$ provided by \eqref{eq:2brelation} and the relation $b_1 = q b_2$, 
\bigskip
\\
We turn our attention to demonstrate \eqref{eq:temp6}. Since $y_k = x_2$, according to the update rule of $\vec{y}(t) = (y_1(t) > \cdots > y_k(t))$ with initial state $\vec{y}$, the only possible case for   $\newH(\vec{x}, \vec{y}(1)) \neq 0$ is either $y_k (1) = x_1$ or $y_k (1) = x_2$. Therefore, we have
\begin{equation}\label{eq:2bye}
\EE^{\vec{y}} \big[\newH(\vec{x}, \vec{y}(1)) \big] = \EE^{\vec{y}} \big[\newH(\vec{x}, \vec{y}(1)) \id_{\left\{y_k(1) = x_2\right\}} \big] + \EE^{\vec{y}} \big[\newH(\vec{x}, \vec{y}(1)) \id_{\left\{y_k(1) = x_1\right\}} \big].
\end{equation}
For the first term on the RHS of \eqref{eq:2bye}, we readily have 
\begin{align*}\numberthis \label{eq:2bye1}
\EE^{\vec{y}} \big[\newH(\vec{x}, \vec{y}(1)) \id_{\left\{y_k (1) = x_2\right\}}\big] = q^{-k} \EE^{\y} \big[\newH(\xp, \vec{y}(1)) \id_{\left\{y_k (1) = x_2\right\}}\big] = q^{-k} \EE^{\y} \big[\newH(\xp, \vec{y}(1))\big]  = q^{-k} R_1.
\end{align*}
The first equality above is due to the fact that the condition $y_{k} (1) = x_2$ implies $\newH(\vec{x}, \vec{y}(1)) = q^{-k } \newH(\vec{x'}, \vec{y}(1))$.
\bigskip
\\ 
For the second term on the RHS of \eqref{eq:2bye}, we expand the expectation (using the condition $x_k = y_2$) 
\begin{align}\label{eq:2byee}
\EE^{\vec{y}} \big[\newH(\vec{x}, \vec{y}(1)) \id_{\left\{y_k(1) = x_1\right\}} \big]
= \mathbf{A} + \mathbf{B},
\end{align}
where 
\begin{equation*}
\mathbf{A} = \sum_{\substack{\vec{w} = (w_1 > \cdots > w_k)\\ w_k = x_1, w_{k-1} > x_2}} \PP^{\vec{y}}\big(\vec{y}(1) = \vec{w}\big) \newH(\vec{x}, \vec{w}),\qquad
\mathbf{B} = \sum_{\substack{\vec{w} = (w_1 > \cdots > w_k)\\ w_k = x_1, w_{k-1} = x_2}} \PP^{\vec{y}}\big(\vec{y}(1) = \vec{w}\big) \newH(\vec{x}, \vec{w})
\end{equation*}
It is easy to check that given $\vec{w} = (w_1 > \cdots > w_k)$ and $\wp = (w_1 > \cdots > w_{k-1})$ with condition $w_k = x_1$ and $w_{k-1} > x_2$  implies
\begin{equation*}
\PP^{\vec{y}} \big(\vec{y}(1) = \vec{w}\big) = \ronepart(x_2, x_1)  \PP^{\yp} \big(\yp(1) = \wp\big), \qquad \newH(\vec{x}, \vec{w}) = q^{-(2k-1)} \newH(\xpp, \wp).  
\end{equation*}
Therefore, 
\begin{align*}
\mathbf{A} = \PP\big(\vec{y}(1) = \vec{w}\big) \newH(\vec{x}, \vec{w}) &=  q^{-(2k-1)} \ronepart(x_2, x_1) \sum_{\substack{\wp = (w_1 > \cdots > w_{k-1}) \\ w_{k-1} > x_2}}  \PP^{\yp} \big(\yp(1) = \wp\big) \newH(\xpp, \wp),\\
&=q^{-(2k - 1)} (1-b_1) (1-b_2) b_2^{x_2 - x_1 - 1} \EE^{\yp} \big[\newH(\xpp, \yp(1))\big],\\ 
\end{align*}
Using $\EE^{\yp} \big[\newH(\xpp, \yp(1))\big] = R_3$,
\begin{align*}
\numberthis \label{eq:2byee1}
\mathbf{A} = q^{-(2k - 1)} (1-b_1) (1-b_2) b_2^{x_2 - x_1 - 1} R_3.
\end{align*}
Similarly, given $\vec{w} = (w_1 > \cdots > w_k)$ and $\wp = (w_1 > \cdots > w_{k-1})$ with  $w_k = x_1$ and $w_{k-1} = x_2$, we have 
\begin{align*}
\PP^{\vec{y}} \big(\vec{y}(1) = \vec{w}\big)  = b_2^{x_2  - x_1 - 1} \PP^{\yp} \big(\vec{y'}(1) = \wp\big), \qquad \newH(\vec{x}, \vec{w}) = q^{-k} \newH(\xp, \wp).
\end{align*}
Thus, 
\begin{align*}
\mathbf{B} = \sum_{\substack{\vec{w} = (w_1 > \cdots > w_k)\\ w_k = x_1, w_{k-1} = x_2}} \PP^{\vec{y}}\big(\vec{y}(1) = \vec{w}\big) \newH(\vec{x}, \vec{w})
& = q^{-k} b_2^{x_2 - x_1 - 1}  \sum_{\substack{\wp = (w_1 > \cdots > w_{k-1})\\ w_{k-1} = x_2}} \PP^{\yp}\big(\yp(1) = \wp\big)   \newH(\xp, \wp), \\ 
& = q^{-k} b_2^{x_2 - x_1 -1}  \EE^{\yp} \big[\newH(\xp, \yp(1)) \id_{\left\{y'_{k-1} (1) = x_2\right\}}\big].
\end{align*}
Consequently, one has
\begin{align*}
\numberthis \label{eq:2byeee}
\mathbf{B} = q^{-k} b_2^{x_2 - x_1 -1}  \bigg(R_2 - \EE^{\yp} \big[\newH(\xp, \yp(1)) \id_{\left\{y'_{k-1} (1) > x_2\right\}}\big]  \bigg).
\end{align*}
Under event $\left\{y'_{k-1} (1) > x_2\right\}$, we have $\newH(\xp, \yp(1)) = q^{-(k-1)} \newH(\xpp, \yp(1))$ and hence
 $$\EE^{\yp} \big[\newH(\xp, \yp(1)) \id_{\left\{y'_{k-1} (1) > x_2\right\}} \big] = q^{-(k-1)} \EE^{\yp} \big[\newH(\xpp, \yp(1)) \id_{\left\{y'_{k-1} (1) > x_2\right\}} \big] = q^{-(k-1)} R_3.$$
We obtain from \eqref{eq:2byeee} that 
\begin{align*}\numberthis \label{eq:2byee2}
\mathbf{B}
= q^{-k} b_2^{x_2 - x_1 - 1} (R_2 - q^{-(k-1)} R_3).
\end{align*}
Recall \eqref{eq:2byee} that $\EE^{\vec{y}} \big[\newH(\vec{x}, \vec{y}(1)) \id_{\left\{y_k(1) = x_1\right\}} \big]
= \mathbf{A} + \mathbf{B}$, using \eqref{eq:2byee1} and \eqref{eq:2byee2}, we get
\begin{equation}\label{eq:2bye2}
\EE^{\vec{y}} \big[\newH(\vec{x}, \vec{y}(1)) \id_{\left\{y_k(1) = x_1\right\}} \big]  = q^{-(2k - 1)} (1-b_1) (1-b_2) b_2^{x_2 - x_1 - 1} R_3 + q^{-k} b_2^{x_2 - x_1 - 1} (R_2 - q^{-(k-1)} R_3).
\end{equation}
Note that the LHS of \eqref{eq:2bye1} and \eqref{eq:2bye2} are the first and second term on the RHS of \eqref{eq:2bye}, hence  
\begin{align*}
\EE^{\vec{y}} \big[\newH(\vec{x}, \vec{y}(1)) \big] &= q^{-k} R_1  +  q^{-(2k - 1)} (1-b_1) (1-b_2) b_2^{x_2 - x_1 - 1} R_3 + q^{-k} b_2^{x_2 - x_1 - 1} (R_2 - q^{-(k-1)} R_3),\\ 
&= q^{-k} R_1 + b_2^{x_2 - x_1 -1} (q^{-k} R_2 + q^{-(2k-1)} (b_1 b_2 - b_1 - b_2)R_3).
\end{align*}
We have proved the desired \eqref{eq:temp6}, thus concluding \eqref{eq:needtoshoww} for the case $x_2 = y_k$.
\bigskip
\\
It only remains to prove \eqref{eq:needtoshoww} for the following case:
\bigskip
\\
\textbf{Case (2c):} $y_k > x_2 > x_1$.
\bigskip
\\
The computation for this case is similar to Case (2b). Let us consider the S6V location processes 
\begin{align*}
&\xp(t) = (\xpr_1(t) < \cdots < x'_{\ell - 1} (t) )  \, \text{ with initial state } \xp = (x_2 < \cdots < x_\ell),\\
&\xpp(t) = \leftvec{x''_1 (t)}{x''_{\ell-2} (t)} \text{ with initial state } \xpp = \leftvec{x_3}{x_\ell},
\end{align*}
and denote
\begin{align*}
&L_1 = \EE^{\xp} \big[\newH(\xp(1), \y)\big], \qquad L_2 = \EE^{\xpp} \big[\newH(\xpp(1), \y)\big],\\
&R_1 = \EE^{\y} \big[\newH(\xp, \vec{y}(1))\big], \qquad R_2 = \EE^{\y} \big[\newH(\xpp, \vec{y}(1))\big].
\end{align*} 
By \eqref{eq:induchypo}, we have 
\begin{equation}\label{eq:2cidhypo}
L_1 = R_1, \qquad L_2 = R_2.
\end{equation}
To conclude \eqref{eq:needtoshoww}, it suffices to show that 
\begin{align}\label{eq:2cx}
\EE^{\vec{x}} \big[\newH(\vec{x}(1), \vec{y})  \big]  
&= q^{-k} L_1 + q^{-(k-1)} b_2^{x_2 - x_1} \big(L_1 - q^{-k} L_2\big),\\
\label{eq:2cy}
\EE^{\vec{y}} \big[\newH(\vec{x}, \vec{y}(1))  \big] 
&= q^{-k} R_1 + q^{-(k-1)} b_2^{x_2 - x_1}\big(R_1 - q^{-k} R_2\big).
\end{align}
To prove \eqref{eq:2cx}, we first write
\begin{align}\label{eq:2cxe}
\leftexp = \EE^{\vec{x}} \big[\newH(\vec{x}(1), \vec{y})  \id_{\left\{x_1 (1) < x_2\right\}}  \big] + \EE^{\vec{x}} \big[\newH(\vec{x}(1), \vec{y})  \id_{\left\{x_1 (1) = x_2\right\}}  \big].
\end{align}
Similar as \eqref{eq:2bfirst}, the first term on the RHS of \eqref{eq:2cxe} can be expressed as  
\begin{align*}\numberthis \label{eq:2cxe1}
\EE^{\vec{x}} \big[\newH(\vec{x}(1), \vec{y})  \id_{\left\{x_1 (1) < x_2\right\}}  \big] 
= q^{-k} (1 - (1-b_1) b_2^{x_2 - x_1 - 1}) L_1,
\end{align*}
while the second term on the RHS of \eqref{eq:2cxe} equals 
\begin{align*} \numberthis \label{eq:2ctemp}
\EE^{\vec{x}} \big[\newH(\vec{x}(1), \vec{y})  \id_{\left\{x_1 (1) = x_2\right\}}  \big] &= \sum_{\substack{\vec{z} = (z_1 < \cdots < z_\ell) \\ z_1 = x_2}} \PP^{\vec{x}}\big(\vec{x}(1) = \vec{z}\big) \newH(\vec{z}, \vec{y}).
\end{align*}
Given $\vec{z} = \leftvec{z_1}{z_\ell}$ and $\zp = \leftvec{z_2}{z_\ell}$ with $z_1 = x_2$, we have 
\begin{equation*}
\PP^{\vec{x}}\big(\vec{x}(1) = \vec{z}\big) = b_2^{x_2 - x_1 - 1} \PP^{\xp}\big(\xp(1) = \zp\big), \qquad \newH(\vec{z}, \vec{y}) = q^{-k} \newH(\zp, \y).
\end{equation*}
Substituting back to \eqref{eq:2ctemp} yields
\begin{align*}
\EE^{\vec{x}} \big[\newH(\vec{x}(1), \vec{y})  \id_{\left\{x_1 (1) = x_2\right\}}  \big] &= q^{-k} b_2^{x_2 - x_1 - 1} \sum_{\substack{\vec{z'} = (z_2 < \cdots < z_\ell) \\ z_2 > x_2 }} \PP^{\xp}\big(\xp(1) = \zp\big)  \newH(\zp, \y),\\ 
&= q^{-k} b_2^{x_2 - x_1 - 1} \EE^{\xp} \big[\newH(\xp(1), \y) \id_{\left\{x'_1 (1) > x_2\right\}}\big].
\end{align*}
Consequently, 
\begin{align*}
\numberthis \label{eq:2cxee1}
\EE^{\vec{x}} \big[\newH(\vec{x}(1), \vec{y})  \big]  = q^{-k} b_2^{x_2 - x_1 - 1} \bigg(L_1 - \EE^{\xp} \big[\newH(\xp(1), \y) \id_{\left\{x'_1 (1) = x_2\right\}}\big]\bigg).
\end{align*}
By Lemma \ref{lem:temp}, we have  \begin{equation*}
\EE^{\xp} \big[\newH(\xp(1), \y) \id_{\left\{x'_1 (1) = x_2\right\}}\big] = q^{-k} b_1 \EE^{\xpp} \big[\newH(\xpp(1), \y)\big] = q^{-k} b_1 L_2.
\end{equation*}
Using \eqref{eq:2cxee1}, 
\begin{align*}\numberthis \label{eq:2cxe2}
\EE^{\vec{x}} \big[\newH(\vec{x}(1), \vec{y})  \id_{\left\{x_1 (1) = x_2\right\}}  \big] 
=  q^{-k} b_2^{x_2 - x_1 - 1} \big(L_1 - q^{-k} b_1 L_2 \big).
\end{align*} 
Plugging \eqref{eq:2cxe1} and \eqref{eq:2cxe2} into the RHS of \eqref{eq:2cxe} yields 
\begin{align*}
\EE^{\vec{x}} \big[\newH(\vec{x}(1), \vec{y})  \big] &=  q^{-k} (1 -  (1 - b_1) b_2^{x_2 -  x_1 - 1}) L_1 + q^{-k} b_2^{x_2 - x_1 - 1} \big(L_1 - q^{-k} b_1 L_2 \big), \\ 
&= q^{-k} L_1 + q^{-(k-1)} b_2^{x_2 - x_1} \big(L_1 - q^{-k} L_2\big),
\end{align*}
We conclude \eqref{eq:2cx}. Here, in the last line above we used again the relation $b_1 = q b_2$. 
\bigskip
\\
We turn to demonstrate \eqref{eq:2cy}. As  $y_k(1) < x_1$ implies $\newH(\vec{x}, \vec{y}(1)) = 0$, we have
\begin{align*}
\EE^{\vec{y}} \big[\newH(\vec{x}, \vec{y}(1))  \big] &= \EE^{\vec{y}} \big[\newH(\vec{x}, \vec{y}(1))  \id_{\left\{y_k(1) > x_1\right\}} \big] + \EE^{\vec{y}} \big[\newH(\vec{x}, \vec{y}(1))   \id_{\left\{y_k(1) = x_1\right\}}\big].
\end{align*}
Since $y_{k} (1) > x_1$ implies $\newH(\x, \y(1)) = q^{-k} \newH(\xp, \y(1))$,
\begin{equation*}
\EE^{\vec{y}} \big[\newH(\vec{x}, \vec{y}(1))  \id_{\left\{y_k(1) > x_1\right\}} \big] = q^{-k} \EE^{\vec{y}} \big[\newH(\xp
, \vec{y}(1))  \id_{\left\{y_k(1) > x_1\right\}} \big] = q^{-k} \EE^{\vec{y}} \big[\newH(\xp
, \vec{y}(1))   \big] = q^{-k} R_1.
\end{equation*}
Consequently, 
\begin{align*}
\numberthis \label{eq:2cye1}
\EE^{\vec{y}} \big[\newH(\vec{x}, \vec{y}(1))  \big] 
= q^{-k} R_1 + \EE^{\vec{y}} \big[\newH(\vec{x}, \vec{y}(1))   \id_{\left\{y_k(1) = x_1\right\}}\big].
\end{align*}
For the second term on the RHS of \eqref{eq:2cye1}, we have 
\begin{align*}\numberthis \label{eq:2cyee1}
\EE^{\vec{y}} \big[\newH(\vec{x}, \vec{y}(1))   \id_{\left\{y_k(1) = x_1\right\}}\big] = \sum_{\substack{\vec{w} = (w_1 > \cdots > w_k)\\ w_k = x_1}} \PP^{\vec{y}} \big(\y(1) = \vec{w}\big) \newH(\vec{x}, \vec{w})
=\sum_{\substack{\vec{w} = (w_1 > \cdots > w_k)  \\  w_k = x_1, w_{k-1} \geq y_k}}  \PP^{\y} \big(\y(1) = \vec{w}\big) \newH (\vec{x}, \vec{w}).
\end{align*}
Since $y_k > x_2$, given $\vec{w} = (w_1 > \cdots > w_{k-1} > x_1)$ and $\vec{w'} = (w_1 > \cdots > w_{k-1} > x_2)$ satisfying $w_{k-1} \geq y_k$, we have 
\begin{equation*}
\PP^{\vec{y}} \big(\vec{y}(1) = \vec{w}\big) = b_2^{x_2 -  x_1} \PP^{\y}\big(\vec{y}(1) = \wp\big), \qquad \newH(\vec{x}, \vec{w}) = q^{-(k-1)} \newH(\xp, \wp).
\end{equation*}
Substituting back to the RHS of \eqref{eq:2cyee1} yields
\begin{align*}
\EE^{\vec{y}} \big[\newH(\vec{x}, \vec{y}(1))   \id_{\left\{y_k(1) = x_1\right\}}\big] &= q^{-(k-1)} b_2^{x_2 - x_1} \sum_{\substack{\wp = (w_1 > \cdots > w_k)\\ w_k = x_2, w_{k-1} \geq y_k}} \PP^{\y}\big(\vec{y}(1) = \wp\big) \newH(\xp, \wp),\\
&= q^{-(k-1)} b_2^{x_2 -x_1} \EE^{\y} \big[\newH(\xp, \vec{y}(1)) \id_{\left\{y_k(1) = x_2\right\}}\big].
\end{align*} 
Consequently,
\begin{align*}
\numberthis \label{eq:2cyeee}
\EE^{\vec{y}} \big[\newH(\vec{x}, \vec{y}(1))   \id_{\left\{y_k(1) = x_1\right\}}\big] &=q^{-(k-1)} b_2^{x_2 -x_1} \big(R_1 - \EE^{\y} \big[\newH(\xp, \vec{y}(1)) \id_{\left\{y_k(1) > x_2\right\}}\big]\big).
\end{align*}
Under event $\left\{y_k (1) > x_2\right\}$, we have $\newH(\xp, \y(1)) = q^{-k} \newH(\xpp, \y(1))$ and accordingly
$$\EE^{\y} \big[\newH(\xp, \vec{y}(1)) \id_{\left\{y_k(1) > x_2\right\}}\big] = q^{-k} \EE^{\y} \big[\newH(\xpp, \vec{y}(1) \id_{\left\{y_k(1) > x_2\right\}}) \big] = q^{-k} \EE^{\y} \big[\newH(\xpp, \vec{y}(1)  \big] = q^{-k} R_2.$$
Therefore, we have by \eqref{eq:2cyeee}
\begin{align*}\numberthis \label{eq:2cye2}
\EE^{\vec{y}} \big[\newH(\vec{x}, \vec{y}(1))   \id_{\left\{y_k(1) = x_1\right\}}\big] 
&= q^{-(k-1)} b_2^{x_2  -x_1} R_1 - q^{-(2k-1)} b_2^{x_2  -x_1} R_2. 
\end{align*}
Substituting \eqref{eq:2cye2} back to \eqref{eq:2cye1} entails 
\begin{align*}
\EE^{\vec{y}} \big[\newH(\vec{x}, \vec{y}(1))  \big] &= q^{-k} R_1 + q^{-(k-1)} b_2^{x_2  -x_1} R_1 - q^{-(2k-1)} b_2^{x_2  -x_1} R_2, \\
&= q^{-k} R_1 + q^{-(k-1)} b_2^{x_2 - x_1}\big(R_1 - q^{-k} R_2\big).
\end{align*}
which concludes the desired \eqref{eq:2cy}.
\bigskip
\\
As all the possible cases for $x \in \XX^\ell$ and $\vec{y} \in \YY^k$ were discussed, we justify \eqref{eq:needtoshoww} and conclude Theorem \ref{thm:main}.
\end{proof}

%
%
%
%
%
\bibliographystyle{alpha}
\bibliography{S6Vfinal1.bib}
\end{document}